\documentclass[english, 11pt]{amsart}

\usepackage{color}
\newenvironment{red}{\color{red}}{}
\newcommand{\bred}{\begin{red}}
\newcommand{\ered}{\end{red}}

\usepackage{babel}
\usepackage{amstext}
\usepackage{amsmath}
\usepackage{amsfonts}
\usepackage{latexsym}
\usepackage{ifthen}
\newcommand\sO{{\mathcal O}}

\newcommand\sC{{\mathcal C}}

\newcommand\bZ{{\mathbb Z}}

\newcommand\bP{{\mathbb P}}

\numberwithin{equation}{section}
\newtheorem{theorem}[equation]{Theorem}
\newtheorem*{theorem*}{Theorem}
\newtheorem{proposition}[equation]{Proposition}

\theoremstyle{remark}

\newtheorem*{problem*}{Problem}

\newcommand{\Wedge}[1]{{\textstyle{\bigwedge\nolimits}^{\! #1}}}

\title {Contact Moishezon threefolds with second Betti number one}
\author{Jaros\l{}aw Buczy\'nski \and Thomas Peternell}
\date{March 21, 2012}

\begin{document}


\begin{abstract}
 We prove that the only contact Moishezon threefold having second Betti number equal to one is the projective space.
\end{abstract}

\maketitle

\section{Introduction}

A compact complex manifold $X$ of dimension $2n+1$ is said to be a \emph{contact manifold} if there is a vector bundle sequence
$$ 0 \to F \to T_X \to L \to 0 $$
where $T_X$ is the tangent bundle of $X$ and $F$ a sub-bundle of rank $2n$ such that the induced map
$$\Wedge{2} F \to L = T_X/F, \  v \wedge w \mapsto [v,w]/F $$
is everywhere non-degenerate.
Properly speaking $(X,F)$ is a contact manifold;
 the line bundle $L$ is called the \emph{contact line bundle}.
It is easy to see that  $-K_X =  (n+1) L$, that is in case $X$ is a threefold:
 \begin{equation} \label{equ_minus_KX_eq_2L}
    -K_X = 2 L.
 \end{equation}
We refer e.g.~to \cite{Le95} and \cite{Bu09} for details.
There are basically two methods to construct compact contact manifolds.
\begin{itemize}
\item A simple Lie group gives rise to a Fano contact manifold $X$
      by taking the unique closed orbit for the adjoint
      action of the Lie group on the the projectivised Lie algebra; we refer to \cite{Be98}.
     Unless the group is of type $A$, we have $b_2(X) = 1$.
      Specifically, this construction includes $\bP^{2n+1}$, $\bP(T_{\bP^{n+1}})$, Grassmannians of lines on quadrics,
           and some exceptional homogeneous spaces.
\item given any compact complex manifold $M$, the projectivised tangent bundle $\bP(T_M)$ is a contact manifold.
\end{itemize}

A famous conjecture of LeBrun and Salamon \cite{LS94} claims there are no other {\it projective} contact manifolds.
If $b_2(X) \geq 2,$ this is settled by \cite{KPSW00} and \cite{De02}. For results in case $b_2(X) = 1, $ we refer to
\cite{Be98}, \cite{Bu09}, \cite{Bu10}, \cite{Ke01}, \cite{Ke05}, \cite{Le95} and \cite{Wi00}.
Since there is no known example of a compact contact manifold not in the above list, one might wonder whether the projectivity
assumption in the conjecture of LeBrun and Salamon is really necessary. Dropping the projectivity assumption, it seems reasonable
to assume first that $X$ is not too wild, i.e.~$X$ is in class $\sC, $ which is to say that $X$ is bimeromorphic to a
compact K\"ahler manifold.

In \cite{FP11} it has been shown that a contact threefold in class $\sC$ which is not rationally connected must be of the form
$X = \bP(T_M)$ with a K\"ahler surface $M.$ Thus it remains in dimension 3 to treat rationally connected varieties in class $\sC.$
Notice that these are automatically Moishezon spaces, i.e., carry three algebraically independent meromorphic functions,
  see Proposition~\ref{prop_rationally_connected_in_class_C}.
In fact, the rational connectedness
of $X$ implies that $H^2(X,\sO_X) = 0.$
In this short note we treat the case that $b_2(X) = 1.$

\begin{theorem}
Let $X$ be a smooth threefold in class $\sC$ with $b_2(X) = 1$, which is contact for some choice of $F \subset T X$.
Then $X \simeq \bP^3$.
\end{theorem}

In the projective case, this theorem has first been shown by \cite{Ye94}.

\section{Preliminaries}

We will make heavily use of the following theorem of Koll\'ar \cite{Ko91} and Nakamura.

\begin{proposition}\label{prop_theorem_Kollar_Nakamura}
 Let $X$ be a smooth Moishezon threefold with ${\rm Pic}(X) \simeq \bZ$ and let $\sO_X(1)$ be the big ( = effective) generator
of ${\rm Pic}(X).$ Write $K_X = \sO_X(m)$ with some integer $m$ and assume $m < 0.$
Then
\begin{enumerate}
\item \label{item_m_ge_minus_4}
  $m \geq -4 $ and $m = -4$ if and only if $X = \bP^3.$
\item $m = -3$ if and only if $X$ is the 3-dimensional quadric.
\item \label{item_m_eq_minus_2}
  If $m = -2$, then $h^0(X,\sO_X(1)) \leq 7.$
\item \label{item_H2_H3_eq_0}
   $H^2(X,\sO_X(1)) = H^3(X,\sO_X(1)) = 0.$
\end{enumerate}
\end{proposition}
For the proof, see \cite[Thm~(5.3.4)]{Ko91}, \cite[Thm~(5.3.12)]{Ko91}, and \cite[Cor.~(5.3.9)]{Ko91}, respectively.

\vskip .2cm
Next we collect some basic properties of rationally connected manifolds. Recall that a compact manifold in class $\mathcal C$ is {\it rationally connected}
if two general points in $X$ can be joined by a chain of rational curves.
For the benefit of the reader we list the following well-known properties and include indications on the proof.

\begin{proposition}\label{prop_rationally_connected_in_class_C}
 Let $X$ be a rationally connected manifold in class $\mathcal C.$ Then the following holds.
\begin{enumerate}
\item $X$ is simply connected;
\item $H^q(X,\mathcal O_X) = 0$ for all $q \geq 1$; in particular $X$ is Moishezon.
\item ${\rm Pic}(X) $ does not have torsion; so if $b_2(X) = 1,$ then ${\rm Pic}(X) = \mathbb Z.$
\end{enumerate}
\end{proposition}

\begin{proof}
(1) We refer to  \cite[Cor.~5.7]{Ca94}. Notice that in \cite{Ca94}, the manifold is supposed to be K\"ahler. Since however $X$ is bimeromorphically equivalent to a K\"ahler manifold, we may choose a birational
holomorphic map $\hat X \to X$ with $\hat X$ K\"ahler, given by a sequence of blow-ups with smooth centers. Then we apply Campana's  theorem on $\hat X$ and use the basic fact $\pi_1(\hat X) = \pi_1(X)$
(it suffices to check that for a single blow-up along a submanifold). \\
(2) Since $X$ is rationally connected, there exists a rational curve $C \subset X$ such that the tangent bundle $T_X \vert C$ is ample, see \cite[IV.3.7]{Ko96}
(the proof works for manifolds in class $\mathcal C$ as well).
From this fact it follows easily
$$ H^0(X,\Omega^q_X) = 0,$$
hence by Hodge duality $H^q(X,\mathcal O_X) = 0$ for $q > 0.$
We refer to \cite[IV.3]{Ko96} for details.  \\
In order to show that $X$ is Moishezon, observe that $H^q(\hat X,\mathcal O_{\hat X}) = 0$ for positive $q$, in particular $H^2(\hat X, \mathcal O_{\hat X}) = 0.$
Thus by Kodaira's classical theorem $\hat X$ is projective and therefore $X$ is Moishezon.
\\
(3) Suppose ${\rm Pic}(X)$ contains a torsion element. Thus there is a non-trivial line bundle $M$ such that $M^{\otimes m} \simeq \mathcal O_X$ for some
positive number $m$. As a consequence, there is a finite \'etale cover $f: \tilde X \to X$ such that $f^*(M) \simeq \mathcal O_{\tilde X}.$ This contradicts the
simply connectedness of $X$.
\end{proof}

\section{Proof of the theorem}

To start the proof of the main theorem, we first  observe that $X$ is uniruled, see \cite[Thm~2.2]{FP11}.
Furthermore, $X$ is rationally connected,
for otherwise by the main Theorem in \cite{FP11}
 $X$ is isomorphic to $\bP(T_M)$
 with a K\"ahler surface $M$ and $b_2(X) \geq 2$, a contradiction to our assumption.
In particular by Proposition 2.2, $X$ is Moishezon,
simply connected and ${\rm Pic }X \simeq \bZ$.
Let $\sO_X(1)$ be the effective generator of ${\rm Pic}(X). $
Since the canonical  line bundle  $K_X$ is divisible by  $2$ by \eqref{equ_minus_KX_eq_2L}, we have
$$ K_X = \sO_X(m)$$
with an even integer $m$.
Since $X$ is uniruled, $m$ must be negative, see \cite[Thms~(5.3.2) and (5.3.3)]{Ko91}.
Applying Proposition~\ref{prop_theorem_Kollar_Nakamura}\ref{item_m_ge_minus_4},
   we simply have to exclude  $m = -2.$
So suppose $m = -2$, in other words the contact line bundle $L =  \sO_X(1)$.
We will arrive at a contradiction with \ref{item_m_eq_minus_2} of Proposition~\ref{prop_theorem_Kollar_Nakamura}
by calculating the number of sections $h^0(X,L)$.

Since $c_3(X)$ is the Euler characteristic of $X$, we have $c_3(X) = b_0 - b_1 + b_2 - b_3 + b_4 - b_5 + b_6$
 with $b_0 = b_6 = 1,$
 $b_1 = b_5 = 0$ ($X$ being simply connected),
 and $b_2 = b_4 = 1$ by our assumption).
 Hence
 \begin{equation}\label{equ_c3X_le_4}
     c_3(X) = 4 - b_3 \leq 4.
 \end{equation}
Since the contact form gives an isomorphism $\Wedge{2}F = L$, we have $c_1(F) = L$.
From the short exact sequece $ 0 \to F \to TX \to L \to 0$ we obtain
$$
  (1+c_1(X) + c_2(X) + c_3(X)) = (1 + L + c_2(F))(1+ L)
$$
In degrees $3$ and $2$ we obtain, respectively:
\begin{align}
 c_3(X) &= c_2(F).L \quad \text{ and } \label{equ_c_3_eq_c_2_F_dot_L}\\
 c_2(X) &= c_2(F) + L^2.               \label{equ_c_2_eq_c_2_F_plus_L_square}
\end{align}
The Riemann-Roch-Hirzebruch \cite[\S XIX.4]{Pa65}
   formula for $\sO_X$ gives:
\[
   \underbrace{\chi(\sO_X)}_{=1 \text{ by Prop.~\ref{prop_rationally_connected_in_class_C}}}
   = \frac{1}{24} \underbrace{c_1(X)}_{= 2L \text{ by Prop.~\eqref{equ_minus_KX_eq_2L}}}.c_2(X)
\]
and thus:
\begin{align}
   &L.c_2(X)  = 12, \text{ and} \label{equ_L_c_2}\\
   &L^3 \stackrel{\text{by \eqref{equ_c_2_eq_c_2_F_plus_L_square}}}{=}
        \underbrace{L.c_2(X)}_{=12 \text{ by~\eqref{equ_L_c_2}}}
        - \underbrace{L.c_2(F)}_{= c_3(X) \text{ by \eqref{equ_c_3_eq_c_2_F_dot_L}}} =
      12 - c_3(X) \stackrel{\text{by \eqref{equ_c3X_le_4}}}{\geq} 8 \label{equ_L_cube}.
\end{align}
Now Riemann-Roch-Hirzebruch for $L$ reads:
\begin{align*}
 \chi(L) & = \frac{1}{3!} L^3 +  {\frac{1}{2}} L^2. \underbrace{\frac{-K_X}{2}}_{=L \text{ by \eqref{equ_minus_KX_eq_2L}}}
                             +  L.\frac{(-{K_X})^2 + c_2(X)}{12} + \underbrace{\chi(\mathcal O_X)}_{=1 \text{ by Prop.~\ref{prop_rationally_connected_in_class_C}}}  =\\
         & = \frac{1}{6} L^3 +  {\frac{1}{2}} L^3
                             + {\frac{1}{3}} L. \underbrace{\left(\frac{-K_X}{2}\right)^2}_{=L^2
                                                         \text{ by \eqref{equ_minus_KX_eq_2L}}}
                              + \underbrace{\frac{L.c_2(X)}{12}}_{= 1 \text{ by \eqref{equ_L_c_2}}} + 1 =\\
         & = \frac{1}{6} L^3 +  {\frac{1}{2}} L^3 + \frac{1}{3}L^3 + 1 + 1 =\\
         & = L^3  + 2  \geq 10 \text{ (by \eqref{equ_L_cube})}.
\end{align*}
Since $h^2(L) = h^3(L)= 0$ by Proposition~\ref{prop_theorem_Kollar_Nakamura}, \ref{item_H2_H3_eq_0},
   we have $h^0(L) = \chi(L) + h^1(L) \geq 10$.
This contradicts part \ref{item_m_eq_minus_2} of Proposition~\ref{prop_theorem_Kollar_Nakamura}.

\subsection*{Acknowledgements}
The first named author was supported by a Maria Sk\l{}o\-do\-wska-Curie Outgoing Fellowship "Contact Manifolds." He also
would like to thank the University of Bayreuth for invitation, support of his visit and providing a nice and stimulating atmosphere for research.

\vskip 1cm

{\small
\begin{center}
\noindent\begin{tabular}{lcl}
Jaros\l{}aw Buczy\'nski && Thomas Peternell \\
\verb|jabu@mimuw.edu.pl| && \verb|thomas.peternell@uni-bayreuth.de|\\
Institut Fourier, Univ.~Grenoble I  && Mathematisches Institut \\
100~rue des Maths, BP~74 && Universit\" at Bayreuth\\
38402 St~Martin d'H\`eres, France && D-95440 Bayreuth, Germany \\
\textbf{and}&& \\
 Institute of Mathematics&& \\
 Polish Academy of Sciences && \\
ul. \'Sniadeckich 8,  P.O. Box 21,&& \\
00-956 Warszawa, Poland &&\\
\end{tabular}
\end{center}}

\end{document}